\documentclass[a4paper, reqno, 12pt]{amsart}
\usepackage[utf8]{inputenc} 
\usepackage[T1, T2A]{fontenc}
\usepackage[english]{babel}
\usepackage{amsmath, amssymb, amsfonts, amsthm, amscd, mathrsfs,stmaryrd}
\usepackage{enumitem}
\usepackage[hidelinks]{hyperref}
\usepackage{tikz-cd}
\setlist[enumerate]{label = (\alph*), ref=(\text{\alph*)}}
\setlist[itemize]{nolistsep}

\usepackage{epsfig}
\usepackage{graphicx}

\renewcommand{\phi}{\varphi}
\renewcommand{\ge}{\geqslant}
\renewcommand{\le}{\leqslant}

\newcommand{\CC}{\mathbb{C}}
\newcommand{\KK}{\mathbb{K}}
\renewcommand{\AA}{\mathbb{A}}
\newcommand{\QQ}{\mathbb{Q}}
\newcommand{\ZZ}{\mathbb{Z}}
\newcommand{\GG}{\mathbb{G}}
\newcommand{\PP}{\mathbb{P}}
\newcommand{\FF}{\mathbb{F}}

\newcommand{\closure}[1]{\overline{#1}}

\newcommand{\pairing}[2]{\langle #1, #2 \rangle}
\newcommand{\reg}{\mathrm{reg}}

\DeclareMathOperator{\Cone}{Cone}

\DeclareMathOperator{\GL}{GL}
\DeclareMathOperator{\SL}{SL}

\DeclareMathOperator{\Sp}{Sp}
\DeclareMathOperator{\Cl}{Cl}

\DeclareMathOperator{\Aut}{Aut}
\DeclareMathOperator{\SAut}{SAut}

\DeclareMathOperator{\PGL}{PGL}

\DeclareMathOperator{\Spec}{Spec}
\DeclareMathOperator{\Hom}{Hom}

\theoremstyle{plain}
\newtheorem{lemma}{Lemma}[section]
\newtheorem{proposition}[lemma]{Proposition}
\newtheorem{theorem}[lemma]{Theorem}
\newtheorem{corollary}[lemma]{Corollary}
\theoremstyle{definition}

\newtheorem{example}[lemma]{Example}
\newtheorem{conjecture}[lemma]{Conjecture}
\newtheorem{question}[lemma]{Question}
\newtheorem{problem}[lemma]{Problem}
\theoremstyle{remark}
\newtheorem{remark}[lemma]{Remark}

\begin{document}

\title[Homogeneous algebraic varieties and transitivity degree]{Homogeneous algebraic varieties and transitivity degree}
\author{Ivan Arzhantsev}
\address{HSE University, Faculty of Computer Science, Pokrovsky Boulevard 11, Moscow, 109028 Russia}
\email{arjantsev@hse.ru}

\author{Kirill Shakhmatov}
\address{HSE University, Faculty of Computer Science, Pokrovsky Boulevard 11, Moscow, 109028 Russia}
\email{bagoga@list.ru}

\author{Yulia Zaitseva}
\address{HSE University, Faculty of Computer Science, Pokrovsky Boulevard 11, Moscow, 109028 Russia}
\email{yuliazaitseva@gmail.com}

\thanks{This research was supported by the Ministry of Science and Higher Education of the
Russian Federation, agreement 075-15-2019-1620 date 08/11/2019 and 075-15-2022-289 date 06/04/2022}



\begin{abstract}
Let $X$ be an algebraic variety such that the group $\Aut(X)$ acts on $X$ transitively. We define the transitivity degree of $X$ as a maximal number $m$ such that the~action of $\Aut(X)$ on $X$ is $m$-transitive. If the action of $\Aut(X)$ is $m$-transitive for all $m$, the transitivity degree is infinite. We compute the transitivity degree for all quasi-affine toric varieties and for many homogeneous spaces of algebraic groups. Also we discuss a~conjecture and open questions related to this invariant. 
\end{abstract}

\maketitle


\section*{Introduction}

Let $X$ be a set and $G$ be a group acting on~$X$. Recall that the action $G \curvearrowright X$ is \emph{transitive} if for any points $x, y \in X$ there is an element $g \in G$ such that $gx=y$. This notion has the following generalization. Fix a positive integer $m$. The action of $G$ on $X$ is \emph{$m$-transitive} if for any pairwise distinct points $x_1, \ldots, x_m \in X$ and any pairwise distinct points $y_1, \ldots, y_m \in X$ there is an element $g \in G$ such that $gx_i = y_i$ for all $1 \le i \le m$. If the action is $m$-transitive for all positive integers $m$, then the action is said to be \emph{infinitely transitive}. 

In this context, it is natural to define the \emph{transitivity degree} $\theta_G(X)$ of an action $G \curvearrowright X$ as the maximal number $m$ such that the action of $G$ on $X$ is $m$-transitive; we let $\theta_G(X)=\infty$ if the action of $G$ on $X$ is infinitely transitive. If the action is not transitive, we let the transitivity degree equal zero. For example, the action of the symmetric group $S_n$ on a set of order $n$ has the transitivity degree~$n$, while the restricted action of the alternating subgroup~$A_n$ is only $(n-2)$-transitive. A generalization of the classical result of Jordan~\cite{Jo1871} based on the classification of finite simple groups claims that there is no other action of a finite group~$G$ on a set~$X$ with $\theta_G(X)>5$; see~\cite[Section~2.1]{DM}.

In this paper we study the transitivity degree of the action of the automorphism group $\Aut(X)$ on an algebraic variety $X$. An algebraic variety $X$ is called a \emph{homogeneous variety} if the group $\Aut(X)$ acts on $X$ transitively. Denote the transitivity degree of the action $\Aut(X) \curvearrowright X$ by~$\theta(X)$ and call it the \emph{transitivity degree} of a variety $X$. Notice that for any regular action $G\times X \to X$ of an algebraic group $G$ we have $\theta_G(X) \le \theta(X)$ since $G$ acts on $X$ by automorphisms. 

The number $\theta(X)$ is an invariant of an algebraic variety $X$. Let us compute this invariant for some varieties. We denote by $\GG_m$ and $\GG_a$ the multiplicative group and the additive group of the base field $\KK$, respectively. We claim that $\theta(\AA^1) = 2$. Indeed, the automorphism group $\Aut(\AA^1) \cong \GG_a \leftthreetimes \GG_m$ consists of affine transformations. Such transformations can move any pair of distinct points to any pair of distinct points, and for triples this does not hold since the intersection of stabilizers of two distinct points is trivial. Hovewer, starting from dimension two, it is well known that the automorphism group $\Aut(\AA^n)$ acts on the affine space $\AA^n$ infinitely transitive, i.e. we have $\theta(\AA^n) = \infty$ for~$n\ge 2$. 

For the projective line~$\PP^1$, the automorphism group $\Aut(\PP^1)$ is $\PGL_2(\KK)$, and any triple of distinct points can be mapped to any triple of distinct points. Indeed, any two distinct points can be mapped to $[1:0], [0:1]$ by sending the corresponding vectors in $\KK^2$ to the basic vectors $(1, 0), (0, 1)$, and the third point can be mapped to $[1:1]$ by a diagonal matrix in $\PGL_2(\KK)$, which stabilizes the pair $[1:0], [0:1]$. On the other hand, projective transformations preserve the cross-ratio of a quadruple of points, so $\theta(\PP^1) = 3$. In contrast to the affine case, in dimensions $n \ge 2$ we have $\theta(\PP^n)=2$ since an automorphism from $\Aut(\PP^n) = \PGL_{n+1}(\KK)$ cannot move a triple of points on a line to a triple of points not on a line. 

From now on we assume that the base field $\KK$ is algebraically closed and of characteristic zero. An algebraic variety $X$ is called a \emph{homogeneous space} if there exists a transitive action of an algebraic group $G$ on $X$. In this case $X$ can be identified with the variety of left cosets $G/H$, where $H$ is the stabilizer in $G$ of a point on $X$. Homogeneous spaces are classical mathematical objects with rich structural theory and many applications; see e.g.~\cite{BSU2013, Gr1997, Hu1975, OV, PV1994, Ti2011}.  

Clearly, any homogeneous space is a homogeneous variety, but not vice versa; the corresponding example is given in Section~\ref{subsec_toric}. It is known that the connected component $\Aut(X)^0$ of the automorphism group of a complete algebraic variety $X$ is an algebraic group~\cite{Ram}. But even in this case the structure of the whole group $\Aut(X)$ is quite mysterious; see e.g. \cite{Brion} for a discussion. In this paper we concentrate mostly on the case of quasi-affine homogeneous varieties. 

Let us describe the content of the paper. We begin with some preliminaries in Section~\ref{prel} and earlier results on infinite transitivity and flexible varieties in Section~\ref{itfv}. In Section~\ref{invfun} we prove that if there is a non-constant invertible regular function on a homogeneous variety~$X$, then $\theta(X)=1$ (Theorem~\ref{theoif}). Also we show that if $X$ is a homogeneous variety with $\KK[X]\ne\KK$ and $X$ is not quasi-affine, then $\theta(X) = 1$ (Proposition~\ref{propaa}). For the last result we need to show that regular functions on an algebraic variety separate points if and only if the variety is quasi-affine. We give a short proof of this fact in Section~\ref{secapp}.  

We apply the results of Section~\ref{invfun} to two particular classes of homogeneous varieties. In Section~\ref{subsec_toric} we deal with toric varieties, and in Section~\ref{homsp} we compute the transitivity degree for many homogeneous spaces of linear algebraic groups. Let us illustrate the second result by examples. If $X$ is the variety of all $(n\times n)$-matrices with determinant 1, then for all $m$ any two $m$-tuples of pairwise distinct matrices are equivalent with respect to the action of $\Aut(X)$. But if $X$ is the variety of all non-degenerate $(n\times   n)$-matrices, then there are non-equivalent pairs of matrices. 

Finally, in Section~\ref{mainconj} we present our main conjecture: if $X$ is a quasi-affine homogeneous variety with $\dim X\ge 2$ and $\KK[X]^{\times}=\KK^{\times}$, then $\theta(X)=\infty$. This conjecture is confirmed in particular cases by the results of previous sections. Also we discuss some properties of homogeneous varieties. Clearly, every homogeneous variety is smooth, but it is easy to give examples of both affine and projective smooth varieties that are not homogeneous. We consider rationality properties of homogeneous varieties and ask whether any homogeneous variety is quasi-projective. 

The authors are grateful to Yuri Prokhorov and Dmitri Timashev for useful discussions. 

\section{Preliminaries}
\label{prel}
In this section we recall some facts and results that will be needed below. 

In \cite{Po1978} all curves with infinite automorphism groups are found. In particular, every such curve which is smooth is either $\PP^1, \AA^1$ or $\AA^1 \setminus \{ 0 \}$. It follows that every homogeneous curve is a homogeneous space and $\theta(\PP^1)=3$, $\theta(\AA^1)=2$, $\theta(\AA^1 \setminus \{ 0 \})=1$. So from now on we assume that $X$ has dimension at least 2. 

Clearly, if $X$ and $Y$ are homogeneous varieties, then the group $\Aut(X)\times\Aut(Y)$ acts on $X\times Y$ transitively, and the variety $X\times Y$ is homogeneous as well. Let us notice that the transitivity degree of direct products is not determined by the transitivity degree of the factors: we have $\theta(\AA^1)=2$ and $\theta(\AA^1\times\AA^1)=\infty$, while for any complete varieties $X$ and $Y$ of positive dimension we expect that $\theta(X\times Y)=1$. An essential argument in favor of the last statement is Blanchard's Lemma claiming that $\Aut(X\times Y)^0=\Aut(X)^0\times\Aut(Y)^0$ for any complete varieties~$X$ and $Y$, see~\cite[Corollary~7.2.3]{Br22017}. Here $\Aut(X)^0$ is the connected component of $\Aut(X)$ in the sense of~\cite{Ram}; see also~\cite{Po2014}.

\smallskip 

Let $X$ be a complete homogeneous space of an algebraic group. Then $X$ is isomorphic to $A\times Y$, where $A$ is an abelian variety and $Y$ is a complete homogeneous space of a linear algebraic group; see~\cite[Theorem~1.3.1]{BSU2013}. In this case $Y$ is a flag variety of a semisimple linear algebraic group. 

\smallskip

Let us give one more preparatory result. 

\begin{lemma}
\label{lemcom}
Let $X=G/H$ be an irreducible homogeneous space of an algebraic group $G$ and $G^0$ be the connected component of unity in $G$. Then $\theta_G(X)=\theta_{G^0}(X)$.
\end{lemma}

\begin{proof}
It suffices to prove that if the action of $G^0$ on $X$ is not $k$-transitive, then the action of $G$ on $X$ is not $k$-transitive too. The first condition means that the group $G^0$
does not act on the open subset $U$ of collections of pairwise distinct points in $X^k$ transitively. Since $G^0$ is normal in $G$, the group $G$ permutes $G^0$-orbits. So if $X^k$ contains an open $G^0$-orbit, such an orbit is unique and so it is invariant under $G$. If there is no open $G^0$-orbit in $X^k$, then the number of $G^0$-orbits in $U$ is infinite. Since the group $G/G^0$ is finite, the group $G$ can not act on $U$ transitively.
\end{proof} 

\begin{remark}
\label{rem1}
It will be useful to prove an analogue of Lemma~\ref{lemcom} in the case when $G$ is the automorphism group $\Aut(X)$ of an algebraic variety $X$ and $G^0$ is the connected component $\Aut(X)^0$. At the moment we have no such proof. 
\end{remark}

The main result of paper~\cite{Kn1983} is the following complete classification of $2$-transitive actions of algebraic groups. 

\begin{theorem} 
\label{tknop}
Let an algebraic group $G$ act effectively on an algebraic variety $X$ with $\dim X \ge 1$. Then we have $\theta_G(X) \ge 2$ exactly in the following cases: 
\begin{enumerate}
  \item the standard action of $G = \PGL_{n+1}(\KK)$ on $X = \PP^n$;
  \item $G = L \rightthreetimes V$, $X = \AA^n$, where $V = \GG_a^n$ and $L$ is one of the following up to a product by a subgroup of the group $\GG_m$ of scalar matrices:
  \begin{itemize}
    \item $\GG_m$ if $n=1$;
    \item $\SL_n(\KK)$ if $\dim V \ge 2$;
    \item $\Sp_{n}(\KK)$ if $\dim V \ge 2$ even;
  \end{itemize}
Such groups $L$ are precisely the linear algebraic groups that act on $V \setminus \{0\}$ transitively. The action is given by the formula $(A, t)(v) = Av + t$, $A \in L, v,t \in V$ after the identification $X \cong V$. 
\end{enumerate}
\end{theorem}

\begin{remark}
\label{rem2}
It is explained at the begining of the proof of~\cite[Satz~2]{Kn1983} that if an algebraic group $G$ acts effectively on a variety $X$ with $\theta_G(X)\ge 2$, then $G$ is a linear algebraic group. 
\end{remark}

It follows that there is no action of an algebraic group $G$ on $X$ with $\theta_G(X) \ge 4$. Moreover, the action of $\PGL_2(\KK)$ on $\PP^1$ is the unique action of an algebraic group with $\theta_G(X) = 3$. 

Let us recall that an $m$-transitive action $G \curvearrowright X$ is \emph{simply $m$-transitive} if the stabilizer in $G$ of an $m$-tuple of pairwise distinct points is trivial. One can check that the unique 3-transitive action $\PGL_2(\KK) \curvearrowright \PP^1$ is simply 3-transitive and the only simply 2-transitive action of a linear algebraic group is the action of $\Aut(\AA^1) = \GG_m \rightthreetimes \GG_a$ on $\AA^1$.

\begin{corollary}
If an algebraic group $G$ acts on an algebraic variety $X$ with $\theta_G(X)\ge 2$, then $X$ is either $\AA^n$ or $\PP^n$. 
\end{corollary}

For a classification of 2-transitive actions of Lie groups, see~\cite{Bor, Kram, Tits}. 

\section{Infinite transitivity and flexible varieties} 
\label{itfv}

Let $X$ be an algebraic variety. Consider a regular action $\GG_a\times X \to X$ of the additive group $\GG_a$ of the ground field $\KK$. The image of $\GG_a$ in the automorphism group $\Aut(X)$ is a \emph{$\GG_a$-subgroup} in $\Aut(X)$. 

We let $\SAut(X)$ denote the subgroup of~$\Aut(X)$ generated by all $\GG_a$-subgroups. Notice that $\SAut(X)$ is a normal subgroup in~$\Aut(X)$.

Denote by $X_{\reg}$ the smooth locus of a variety $X$. We say that a point $x\in X_{\reg}$ is {\it flexible} if the tangent space $T_xX$ is spanned by the tangent vectors to orbits passing through $x$ over all $\GG_a$-subgroups in $\Aut(X)$. The variety $X$ is {\it flexible} if every point $x\in X_{\reg}$ is. Clearly, $X$ is flexible if one point of $X_{\reg}$ is flexible and the group $\Aut(X)$ acts transitively on $X_{\reg}$. Many examples of flexible varieties are given, e.g., in \cite{AFKKZ12013, AFKKZ22013, AKZ2012, APS2014}.

The following result is proved in~\cite[Theorem~0.1]{AFKKZ12013} for affine varieties and is generalized in \cite[Theorem~2]{APS2014} and \cite[Theorem~1.11]{FKZ2016} to the quasi-affine case.  

\begin{theorem} 
\label{main}
Let $X$ be an irreducible quasi-affine variety of dimension $\ge 2$. The following conditions are equivalent.
\begin{enumerate}
\item
The group $\SAut(X)$ acts transitively on $X_{\reg}$.
\item
The group $\SAut(X)$ acts infinitely transitively on $X_{\reg}$.
\item
The variety $X$ is flexible.
\end{enumerate}
\end{theorem}

One may wonder whether flexibility is the only reason for infinite transitivity of the automorphism group. In this direction, there is the following result for quasi-affine varieties admitting a non-trivial action of a connected linear algebraic group. 

\begin{theorem} \cite[Corollary~13]{Ar22018}
\label{ctrans}
Let $X$ be an irreducible quasi-affine variety of dimension $\ge 2$. Assume that $X$ admits a non-trivial $\GG_a$- or $\GG_m$-action. Then the following conditions are equivalent.
\begin{enumerate}
\item
The group $\Aut(X)$ acts 2-transitively on $X$.
\item
The group $\Aut(X)$ acts infinitely transitively on $X$.
\item
The group $\SAut(X)$ acts transitively on $X$.
\item
The group $\SAut(X)$ acts infinitely transitively on $X$.
\end{enumerate}
\end{theorem}

It is an interesting problem to prove Theorem~\ref{ctrans} without the assumption that $X$ admits a non-trivial $\GG_a$- or $\GG_m$-action~\cite[Conjecture~16]{Ar22018}.

\section{Transitivity degree and invertible functions}
\label{invfun}

The following result shows that the existense of a non-constant invertible regular function is an obstacle to multiple transitivity. 

\begin{theorem}
\label{theoif}
Let $X$ be an irreducible homogeneous algebraic variety with a non-constant invertible regular function. Then $\theta(X) = 1$.
\end{theorem}

\begin{proof}
Assume conversely that $\theta(X) \geq 2$ and fix a point $x \in X$. Then the stabilizer $\Aut(X)_x$ of the point $x$ in the group $\Aut(X)$ acts transitively on $Y = X \setminus \{ x \}$.

Denote by $L$ the factor group $\KK[X]^\times / \KK^\times$ of the multiplicative group $\KK[X]^\times$ of invertible regular functions on $X$ modulo non-zero scalars. By \cite[Proposition~1.3]{KKV1989} $L$ is a free abelian group of a finite rank $m$. Fix functions $f_1, \dots, f_m \in \KK[X]^\times$ whose projections form a basis of $L$ and such that $f_1(x) = \dots = f_m(x) = 1$.

Let $\alpha \in \Aut(X)_x$. The dual automorphism $\alpha^* : \KK[X] \to \KK[X]$ induces a lattice isomorphism $\tilde{\alpha} \in \Aut(L) = \GL_m(\ZZ)$. So if $\tilde{\alpha} = (a_{i j})_{i, j = 1}^m$, then we have
$$
g_i := \alpha^*(f_i) = \lambda_i \prod_{j = 1}^m f_j^{a_{j i}}
$$
for some $\lambda_i \in \KK^\times$. Recall that
$$
g_i(x) = f_i(\alpha x) = f_i(x) = 1,
$$
but
$$
g_i(x) = \lambda_i \prod_{j = 1}^m f_j^{a_{j i}}(x) = \lambda_i.
$$
Therefore, $\lambda_i = 1$ for every $i$.

Let $y \in Y$ and denote $t_i = f_i(y)$. We have
$$
f_i(\alpha y) = g_i(y) = \prod_{j = 1}^m t_j^{a_{j i}}.
$$
Since $f_i$ is a non-constant regular function on $Y$, the morphism $f_i : Y \to \AA^1$ is dominant and $f_i(Y)$ is an open subset of $\AA^1$. By our assumption $Y = \Aut(X)_x \cdot y$, hence for each $i$ the subset $\{ \prod_{j = 1}^m t_j^{a_{j i}} \}_{\alpha \in \Aut(X)_x} $ is open in $\AA^1$. It contradicts the following lemma.
\end{proof}

\begin{lemma}
Let $\KK$ be an algebraically closed field of characteristic zero, $t_1, \dots, t_m \in \KK^\times$, and $A = \{ \prod_{i = 1}^m t_i^{a_i} \ | \ a_1, \dots, a_m \in \ZZ \}$. Then the complement $B = \KK \setminus A$ is infinite.
\end{lemma}

\begin{proof}
Let $\FF$ be the subfield $\QQ(t_1, \dots, t_m) \subseteq \KK$. By the primitive element theorem, we have $\FF = \QQ(x_1, \dots, x_k)(u)$ for some $k \geq 0$ and $u, x_1, \dots, x_k \in \KK$, where $u$ is algebraic over $\QQ(x_1, \dots, x_k)$ and $\QQ(x_1, \dots, x_k)$ is a purely transcendental extension of $\QQ$ of transcendence degree $k$.

We have $A \subseteq \FF$, so $B \supseteq \KK \setminus \FF$. Since $\KK$ is algebraically closed, the algebraic closure $\overline{\FF}$ of $\FF$ is contained in $\KK$. But the algebraic closure of $\QQ(x_1, \dots, x_k)$ is not a finite extension of $\QQ(x_1, \dots, x_k)$, so $\overline{\FF} \setminus \FF$ is infinite. Since $B \supseteq \KK \setminus \FF \supseteq \overline{\FF} \setminus \FF$, the set $B$ is also infinite.
\end{proof}

For the next result we need the fact that regular functions separate points on an irreducible algebraic variety $X$ if and only if $X$ is quasi-affine. This is proved in~\cite[Theorem~1.2~(iii) and Theorem~2.1~(4)]{Gr1997} in the case when $X$ is a homogeneous space. We did not find a reference to this fact in complete generality and give a short proof in the Appendix below. 

\begin{proposition}
\label{propaa}
Let $X$ be an irreducible homogeneous algebraic variety. Assume that $X$ is not quasi-affine and $\KK[X]\ne\KK$. Then $\theta(X) = 1$.
\end{proposition}

\begin{proof}
By Proposition~\ref{propqa}, if $X$ is not quasi-affine then there are points $x,y\in X$ such that $f(x)=f(y)$ for any $f\in\KK[X]$. 

On the other hand, the condition $\KK[X]\ne\KK$ implies that there is a point $z\in X$ such that $h(x)\ne h(z)$ for some $h\in\KK[X]$. We conclude that the pair $(x,y)$ can not be moved to the pair $(x,z)$ by an automorphism of $X$, so the action of $\Aut(X)$ on $X$ is not $2$-transitive.
\end{proof}

\begin{corollary}
Let $X$ and $Y$ be irreducible homogeneous algebraic varieties of positive dimension. Assume that $X$ is quasi-affine and $\KK[Y]=\KK$. Then $\theta(X\times Y) = 1$.
\end{corollary} 


\section{Homogeneous toric varieties} 
\label{subsec_toric}

First of all let us briefly recall basic facts on toric varieties and related combinatorics. For a more detailed description see, for example, \cite{CLS2011, Fu1993, Od1988}.

Let $T = \GG_m^n$ be an algebraic torus of rank $n$. A normal irreducible variety $X$ is called \emph{toric}, if there is a faithful action of $T$ on $X$ with an open orbit.

Let $N = \ZZ^n$ be the lattice of one-parameter subgroups of $T$ and $M = \Hom_\ZZ(N, \ZZ) \simeq \ZZ^n$ be the character lattice of $T$. Denote by $N_{\QQ}$ and $M_{\QQ}$ the associated $\QQ$-vector spaces $N \otimes_\ZZ \QQ$ and $M \otimes_\ZZ \QQ$. Let $\pairing{\cdot}{\cdot} : M_\QQ \times N_\QQ \to \QQ$ be the pairing of dual vector spaces $N_\QQ$ and $M_\QQ$.

A \emph{cone} in $N$ is a convex polyhedral cone in $N_{\QQ}$. A cone is called \emph{strongly convex} if it contains no nonzero linear subspace. Recall that the \textit{dual cone} $\sigma^\vee$ to a cone $\sigma$ in $N$ is defined by
$$
\sigma^\vee = \{ u \in M_\QQ \ | \ \pairing{u}{v} \geq 0 \ \forall v \in \sigma \}.
$$
Given a strongly convex cone $\sigma$ in $N$, one can consider the finitely generated $\KK$-algebra $\KK[\sigma^\vee \cap M]$ graded by the semigroup of lattice points of the cone $\sigma^\vee$:
$$
\KK[\sigma^\vee \cap M] = \bigoplus_{u \in \sigma^\vee} \KK \chi^u,
$$
where $\chi^u \cdot \chi^{u'} = \chi^{u + u'}$. Denote by $X(\sigma)$ the corresponding affine variety $\Spec(\KK[\sigma^\vee \cap M])$. It is well-known that given an affine variety $X$ there is a bijection between faithful $T$-actions on $X$ and effective $M$-gradings on $\KK[X]$. The aforementioned $\sigma^\vee \cap M$-grading on $\KK[\sigma^\vee \cap M]$ corresponds to a faithful $T$-action on $X(\sigma)$ with an open orbit. Therefore, $X(\sigma)$ is a toric variety. Moreover, every affine toric variety $X$ arises this way.

A \textit{fan} in $N$ is a finite collection $\Sigma$ of strongly convex cones in $N$ such that for all $\sigma_1, \sigma_2 \in \Sigma$ the intersection $\sigma_1 \cap \sigma_2$ is a face of both $\sigma_1$ and $\sigma_2$, and every face of $\sigma_1$ is an element of $\Sigma$. There is a one-to-one correspondence between toric varieties with an acting torus $T$ and fans in~$N$. Namely, given a fan $\Sigma$ in $N$ the corresponding toric variety $X(\Sigma)$ is a union of affine charts $X(\sigma), \sigma \in \Sigma$, where any two charts $X(\sigma_1)$ and $X(\sigma_2)$ are glued along their open subset $X(\sigma_1 \cap \sigma_2)$, and every toric variety arises this way.

Denote by $|\Sigma|$ the \emph{support} of a fan $\Sigma$ in $N$:
$$
|\Sigma| = \bigcup_{\sigma \in \Sigma} \sigma.
$$
A fan $\Sigma$ in $N$ is \textit{complete} if $|\Sigma| = N_\QQ$. The corresponding toric variety $X(\Sigma)$ is a complete variety if and only if the fan $\Sigma$ is complete.

\begin{example}
If $n = 1$, then every strongly convex cone in $N = \ZZ$ is either $\QQ_{\ge 0}, \QQ_{\le 0}$, or~$\{0\}$, and any fan should be built from them. It follows that there are only 3 toric curves: $\KK^\times, \AA^1$, and $\PP^1$.
\end{example}

One-dimensional faces of a cone in $N$ are called \textit{rays}. Denote by $\Sigma(1)$ the set of rays of all cones in $\Sigma$. There is a one-to-one correspondence between cones $\sigma \in \Sigma$ and $T$-orbits $O(\sigma)$ in $X$ such that $\dim O(\sigma) = n - \dim \langle \sigma \rangle_\QQ$. In particular, each ray $\rho \in \Sigma(1)$ corresponds to a prime $T$-invariant Weil divisor $D_\rho = \closure{O(\rho)}$ on $X$.

\medskip

Assume that there is another fan $\Sigma'$ in a lattice $N'$ and consider the corresponding toric variety $X(\Sigma')$ with the acting torus $T'$. Take a homomorphism $\phi : N \to N'$ with the following property: for every $\sigma \in \Sigma$ there is $\sigma' \in \Sigma'$ such that $\phi_\QQ(\sigma) \subseteq \sigma'$, where $\phi_\QQ : N_\QQ \to N'_\QQ$ is the linear extension of $\phi$. From the fact that $\phi(\sigma) \subseteq \sigma'$ it follows that the dual homomorphism $\phi^* : M' \to M$ maps $(\sigma')^\vee$ to $\sigma^\vee$. It means that there is a homomorphism $\KK[(\sigma')^\vee] \to \KK[\sigma^\vee]$ and the induced morphism $X(\sigma) \to X(\sigma')$. These morphisms patch together to a morphism $\phi : X(\Sigma) \to X(\Sigma')$ which is called \emph{toric}. Toric morphisms are characterized by the following property: if we identify the open torus orbits in $X(\Sigma)$ and $X(\Sigma')$ with tori $T$ and $T'$ correspondingly, then $\phi(T) \subseteq T'$ and $\phi |_T : T \to T'$ is a homomorphism of algebraic groups.

\begin{example}
\label{seeex}
Let $X(\Sigma)$ be a quasi-affine toric variety. We claim that there is an affine toric variety $X(\omega)$ such that there is an open toric embedding $X(\Sigma) \subseteq X(\omega)$ and the codimension of $X(\omega) \setminus X(\Sigma)$ in $X(\omega)$ is greater than or equal to 2. It means that the fan $\Sigma$ consists of some faces of the cone $\omega$ and $\Sigma(1) = \omega(1)$.

Indeed, by~\cite[Theorem~1.6]{PV1994} the variety $X(\Sigma)$ can be embedded as an open toric subset in an affine toric variety $X(\omega')$. Since the closures of $T$-invariant prime divisors on $X(\Sigma)$ are $T$-invariant prime divisors on $X(\omega')$, the set $\Sigma(1)$ is a subset of $\omega'(1)$. Let $\omega$ be the subcone in $\omega'$ generated by all rays in $\Sigma(1)$. Every cone in $\Sigma$ is a face of $\omega'$ and is generated by some rays in $\Sigma(1)$. So every cone in $\Sigma$ is a face of $\omega$ and the open embedding $X(\Sigma) \subseteq X(\omega)$ is the desired one. 
\end{example}

A toric variety $X = X(\Sigma)$ is called \emph{degenerate} if there exists a non-constant invertible regular function on $X$.  By~\cite[Proposition~3.3.9]{CLS2011}, a toric variety $X$ is degenerate if and only if it is isomorphic to the direct product $Y \times \KK^\times$ for some toric variety $Y$ or, equivalently, the rays in $\Sigma(1)$ do not span $N_{\QQ}$. It implies
that for any $X$ there is a non-degenerate toric variety $X'$ such that $X = X' \times (\KK^\times)^k$ for some $k \ge 0$. In terms of the fan $\Sigma$ it means that $\dim \langle |\Sigma| \rangle_\QQ = n - k$ and $X'$ is given by the fan $\Sigma$ considered in the lattice generated by $|\Sigma| \cap N$. By Theorem~\ref{theoif}, for any degenerate homogeneous toric variety $X$ we have $\theta(X) = 1$.

\medskip

Finally, we recall a canonical quotient realization of toric varieties, see e.g. \cite{ADHL2015, Cox1995}. Let $X = X(\Sigma)$ be a non-degenerate toric variety. Denote $\Sigma(1) = \{\rho_1, \dots, \rho_m\}$ and $D_{\rho_i} = D_i$. For each $i = 1, \dots, m$ let $p_i$ be the generator of the semigroup $\rho_i \cap N$. The divisor class group $\Cl(X)$ of $X$ is generated by classes $[D_1], \dots, [D_m]$. Moreover, there is a short exact sequence
$$
\begin{tikzcd}
0 \arrow[r] & M \arrow[r] & \ZZ^m \arrow[r] & \Cl(X) \arrow[r] & 0
\end{tikzcd}
$$
where the map $M \to \ZZ^m$ is given by $u \mapsto (\pairing{u}{p_1}, \dots, \pairing{u}{p_m})$ and the map $\ZZ^m \to \Cl(X)$ is given by $(a_1, \dots, a_m) \mapsto a_1 [D_1] + \dots + a_m [D_m]$. Denote by $G$ the quasitorus $\Hom_{\ZZ}(\Cl(X), \GG_m)$. Applying the $\Hom_\ZZ(-, \GG_m)$-functor to the short exact sequence, we obtain an inclusion $G \hookrightarrow (\KK^\times)^m$.

Consider a $\Cl(X)$-graded homogeneous coordinate ring or the Cox ring $R(X) = \KK[x_1, \dots, x_m]$ with $\deg(x_i) = [D_i]$. The $\Cl(X)$-grading on $R(X)$ gives rise to an action of $G$ on $\KK^m = \Spec R(X)$. Denote by $Z$ the closed subset of $\KK^m$ defined by equations
$$
\prod_{\rho_l \not\in \sigma(1)} x_l = 0 \quad \text{for all} \quad \sigma \in \Sigma.
$$
The subset $Z$ is invariant under the action of the group $G$, thus there is an action of $G$ on the variety $\hat{X} = \KK^m \setminus Z$. The variety $X$ is isomorphic to the good categorical quotient $\hat{X} /\!/ G$ and the isomorphism $X \simeq \hat{X} /\!/ G$ is toric: $T \simeq (\KK^{\times})^m / G$.

\smallskip

Let us consider the problem of description of homogeneous toric varieties and computation of their transitivity degree. It is well known (see e.g. \cite[Theorem~3.9]{Ba2013}) that a complete toric variety is homogeneous if and only if  $X$ is isomorphic to the product of projective spaces $\PP^{n_1} \times \dots \times \PP^{n_k}$ for some positive integers $n_1, \ldots, n_k$. Due to Theorem~\ref{tknop}, we have $\theta(X) \ge 2$ if and only if $k = 1$, in which case $X = \PP^n$, so $\theta(\PP^n) = 2$ if $n\ge 2$ and $\theta(\PP^1) = 3$.

Every smooth affine toric variety is isomorphic to the product $\AA^n\times (\KK^\times)^k$. All these varieties are homogeneous and we already know their transitivity degree. 

At the same time, the class of quasi-affine smooth toric varieties is quite rich. From~\cite{AKZ2012, FKZ2016} we deduce the following result.

\begin{theorem}
\label{toric}
\begin{enumerate}
Let $X$ be a quasi-affine toric variety.
\item
If $X$ is smooth then $X$ is homogeneous.
\item
If $X$ is non-degenerate then $X$ is flexible.
\item
If $X$ is smooth, non-degenerate, and $\dim X \ge 2$, then $\theta(X) = \infty$.
\end{enumerate}
\end{theorem}

\begin{proof}
Every quasi-affine toric curve is either $\AA^1$ or $\KK^\times$, so the first two assertions hold in case $\dim X = 1$. Assume that $\dim X \ge 2$. Let $Y$ be an affine toric variety such that there is an open toric embedding $X \subseteq Y$ and the codimension of $Y \setminus X$ in $Y$ is greater than or equal to 2; see Example~\ref{seeex}. 

If $X$ is non-degenerate, then so is $Y$, and by \cite[Theorem 2.1]{AKZ2012} the variety $Y$ is flexible. Due to \cite[Theorem 1.6]{FKZ2016} we conclude that $X$ is also flexible, so the second assertion is proved.

If $X$ is smooth, then $X = X' \times (\KK^\times)^k$, where $X'$ is a smooth non-degenerate quasi-affine toric variety and $k \ge 0$. Therefore, it is sufficient to prove the first assertion for non-degenerate varieties. But now both the first and the third assertions follow from the second assertion and Theorem~\ref{main}.
\end{proof}

To sum up, for a smooth non-degenerate quasi-affine toric variety $X$ we have $\theta(X) = 2$ if $\dim X = 1$ and $\theta(X) = \infty$ otherwise. 

It is an open problem to find a criterion for a toric variety $X(\Sigma)$ to be homogeneous in terms of the fan $\Sigma$. Toric varieties that are homogeneous spaces of reductive groups are described in~\cite{AG2010}. A conjectural description of all homogeneous toric varieties is given in~\cite[Conjecture~1]{Ar12018}. A construction given in~\cite{Ar12018} provides a wide class of homogeneous toric varieties.  

\smallskip 

Let us proceed with an example of a homogeneous variety that is not a homogeneous space taken from~\cite[Example~2.2]{AKZ2012}. Consider an affine toric surface $X$ corresponding to a cone $\sigma = \Cone(v_1, v_2)$ in a lattice $N = \ZZ^2$, where $v_1$ and $v_2$ are primitive vectors in $N$. Since every primitive vector in $N$ can be supplemented to a basis of $N$, we may assume that $v_1 = (1, 0)$ and $v_2 = (a, b)$ for some $a, b \in \ZZ$. Applying an automorphism of $N$ given by a matrix of the form
$$
\begin{pmatrix}
1 & c \\
0 & \pm1
\end{pmatrix}
$$
we may further assume that $b >a\ge 0$. If $a = 0$ then $b = 1$ and $X = \AA^2$ is a homogeneous space. Let $a > 0$ and denote $X = X_{a, b}$. Applying the quotient presentation, we realize $X$ as the quotient $\AA^2 / C_b$, where $C_b$ is the cyclic group of order $b$ and, given a primitive $b$th root of unity $\zeta \in \KK$, the action of $C_b = \langle \zeta \rangle$ on $\AA^2$ is defined by
$$
\zeta \cdot (x, y) = (\zeta^a x, \zeta y).
$$
By Theorem~\ref{toric} the smooth locus $X_\reg=X\setminus \{P\}$, where $P$ is the $T$-fixed point on $X$, is a homogeneous variety. However, it follows from known classifications that $X_\reg$ is a homogeneous space if and only if $a = 1$; see~\cite{FZ2005} and references therein. 

If $a = 1$ then $X_\reg$ is isomorphic to the homogeneous space $\SL_2(\KK)/ H$, where
$$
H = \left\{
\left. \begin{pmatrix}
\epsilon & \alpha \\
0 & \epsilon^{-1}
\end{pmatrix} \right| \epsilon^b = 1, \ \alpha \in \KK
\right\}.
$$
Note also that $X_{a,b} \simeq X_{a',b'}$ if and only if $b' = b$ and $a=a'$ or $a a' \equiv 1 \mod b$.

Consider the special case $a = b - 1$. One may check that if we take the dual basis in $M$, then the semigroup $\sigma^\vee \cap M$ is generated by vectors $(0, 1), \ (1, 1)$, and $(b, b - 1)$. It means that $X_{b - 1, b}$ is isomorphic to the surface in $\AA^3$ given by equation $z^b = xy$.

\begin{question}
Consider the hypersurface $X(n, b)$ given in $\AA^{n + 1}$ by equation $x_{n+1}^b = x_1 \cdots x_n$. This is an affine toric variety, whose dual cone in a lattice $\ZZ^n$ with basis $e_1, \dots, e_n$ equals 
$$
\sigma^\vee = \Cone(e_1, \dots, e_{n - 1}, (b - 1) \sum_{i = 1}^{n - 1} e_i + b e_n).
$$
For which $n \ge 3, b \ge 2$ the smooth locus $X(n , b)_\reg$ (which is always homogeneous) is a homogeneous space?
\end{question}

\smallskip

It is a natural problem to describe other classes of homogeneous varieties which are not homogeneous spaces, compare with~\cite[Problem~2]{Ar22018}. Especially interesting task is to construct an affine homogeneous variety that is not a homogeneous space; such an example does not exist for toric varieties. 

\section{Homogeneous spaces of linear algebraic groups}
\label{homsp}

Let $X$ be an irreducible homogeneous space of a linear algebraic group $G$. We may assume that $G$ is connected. It is well known that the group $G$ is a semidirect product
$G^{\text{red}}\rightthreetimes R_u(G)$, where $G^{\text{red}}$ is a connected reductive subgroup and $R_u(G)$ is the unipotent radical of $G$; see~\cite[Theorem~6.4]{OV}. 

Let $G^{\text{s}}$ be the intersection of kernels of all characters of the group $G$. Then $G^{\text{s}}=G^{\text{ss}}\rightthreetimes R_u(G)$, where $G^{\text{ss}}$ is the maximal semisimple subgroup in $G^{\text{red}}$. The factorgroup $G/G^{\text{s}}$ is an algebraic torus $T$. In fact, $T$ is isomorphic to the quotient of the central torus in $G^{\text{red}}$
by a finite subgroup, which is the center of $G^{\text{ss}}$. 

We say that a homogeneous space $X$ is of the \emph{first type} if the subgroup $G^{\text{s}}$ acts on $X$ transitively. All other homogeneous spaces are of the \emph{second type}.

\begin{lemma}
\label{uslem}
A homogeneous space $X$ is of the second type if and only if there is a non-constant invertible regular function on $X$. 
\end{lemma} 

\begin{proof}
Let $H$ be the stabilizer in $G$ of a point on $X$. Then the variety $X$ is isomorphic to $G/H$, and the algebra of regular functions $\KK[X]$ may by identified with the algebra $\KK[G]^H$ 
of regular functions on $G$ that are invariant with respect to the action of $H$ on $G$ by right translations. 

Let us consider the restriction $\varphi\colon H\to T$ of the projection $G\to T$. If the homomorphism $\varphi$ is surjective, then for any $g\in G$ we can find an element $h\in H$
whose image in $T$ coincides with the projection of $g$. Equivalently, the element $gh^{-1}$ is contained in $G^{\text{s}}$. Since $(gh^{-1})(eH)=gH$, we conclude that the space $X$ is of the first type. In this case the pullback of an invertible regular function on $G/H$ to $G^s$ is an invertible regular function on $G^s$. By~\cite[Proposition~1.2]{KKV1989}, such a function is proportional to a character of $G^s$. Since any character of $G^s$ is trivial, the corresponding invertible function is a nonzero constant.  

Now assume that the homomorphism $\varphi$ is not surjective. Then its image is contained in the kernel of a non-trivial character $\lambda$ of the torus $T$. The pullback of $\lambda$ to
$G$ is an $H$-invariant regular function, which may be regarded as a non-constant invertible regular function on $X$. Such a function is $G^{\text{s}}$-invariant, so $X$ is of the second type. 
\end{proof}

Now we are ready to compute the transitivity degree of many homogeneous spaces. The following proposition is a reformulation of \cite[Proposition~5.4]{AFKKZ12013}. 

\begin{proposition}
\label{homft}
If a homogeneous space $X$ of the first type is a quasi-affine variety of dimension $\ge 2$, then $\theta(X)=\infty$.
\end{proposition}

\begin{proof}
The group $G^{\text{s}}$ is generated by $\GG_a$-subgroups, see~\cite[Lemma~1.1]{Po2011}. By Lemma~\ref{uslem}, the group $G^{\text{s}}$, and thus the group $\SAut(X)$, acts on $X$ transitively. The claim now follows from Theorem~\ref{main}. 
\end{proof}

\begin{proposition}
\label{homst}
If $X$ is a homogeneous space of the second type, then $\theta(X)=1$.
\end{proposition}

\begin{proof}
It follows from Theorem~\ref{theoif} and Lemma~\ref{uslem}. 
\end{proof}

\begin{example}
Let $X$ be a homogeneous space of the group $\GL_n(\KK)$. If $X$ is not homogeneous with respect to the subgroup $\SL_n(\KK)$, then $\theta(X)=1$. If $X$ is homogeneous with respect to $\SL_n(\KK)$ and quasi-affine, then $\theta(X)=\infty$. 
\end{example}

Let us introduce two important classes of algebraic subgroups. One says that an algebraic subgroup $H$ of a linear algebraic group $G$ is \emph{observable} if the homogeneous space $G/H$ is a quasi-affine variety. A group-theoretic criterion for a subgroup $H$ to be observable was obtained by Sukhanov; see e.g.~\cite[Theorem~7.3]{Gr1997}.  For example, all reductive subgroups $H$ and all subgroups $H$ without non-trivial characters are observable. 

Further, an algebraic subgroup $H$ of a linear algebraic group $G$ is \emph{epimorphic} if we have $\KK[G/H]=\KK$. It is well known that the homogeneous space $G/H$ is complete if and only if $H$ is parabolic. So any parabolic subgroup in $G$ is epimorphic. In fact, the class of epimorphic subgroups is much wider, and it is an open problem to find a group-theoretic criterion for a subgroup $H$ to be epimorphic; see~\cite[Section~1.23]{Gr1997} for discussion, partial results and examples and~\cite{Br12017} for recent results. 

Assume that $H$ is not epimorphic. If $H$ is observable, then $\theta(G/H)$ can be computed by applying Proposition~\ref{homft} and Proposition~\ref{homst}. If $H$ is not observable, then $\theta(G/H)=1$ by Proposition~\ref{propaa}. So, the only remaining case is given in the following problem.

\begin{problem}
Let $G$ be a connected linear algebraic group and $H$ be an epimorphic subgroup of $G$. Find $\theta(G/H)$. 
\end{problem}

\section{The main conjecture and related questions}
\label{mainconj}

We are ready to formulate the main conjecture of this paper. It claims that for a quasi-affine variety $X$ satisfying some mild necessary conditions homogeneity implies infinite transitivity of the automorphism group. 

\begin{conjecture}
\label{con}
Let $X$ be an irreducible homogeneous quasi-affine variety with ${\dim X\ge 2}$ and $\KK[X]^{\times}=\KK^{\times}$. Then $\theta(X)=\infty$. 
\end{conjecture}

This conjecture holds for wide classes of varieties. For example, by Theorem~\ref{toric}~(c), Conjecture~\ref{con} is true for toric varieties. Lemma~\ref{uslem} and Proposition~\ref{homft} show that Conjecture~\ref{con} holds for homogeneous spaces of algebraic groups. Moreover, Theorem~\ref{ctrans} claims that if an irreducible quasi-affine variety $X$ of dimension $\ge 2$ admits a non-trivial $\GG_a$- or $\GG_m$-action and $\theta(X)\ge 2$, then $\theta(X)=\infty$. 

\smallskip

Let us proceed with some possible specific properties of homogeneous varieties. We begin with properties related to rationality. Let us recall that an irreducible variety $X$ is \emph{rational} if the field of rational functions $\KK(X)$ is a purely transcendental extension of the base field~$\KK$. Equivalently, $X$ is rational if and only if $X$ contains an open subset isomorphic to an open subset of some affine space. In particular, every toric variety is rational. Also any linear algebraic group is a rational affine variety. 

An irreducible variety $X$ is \emph{stably rational} if some purely transcendental extension $\KK(X)(y_1,\ldots,y_m)$ over $\KK(X)$ is purely transcendental over $\KK$. Geometrically it means
that for some positive integer $m$ the direct product $X\times\AA^m$ is rational. Clearly, any rational variety is stably rational. In \cite[Example~1.22]{Po2011}, examples of homogeneous spaces of the group $\SL_n(\KK)$ that are not stably rational are given. In particular, not every flexible variety is stably rational.

An irreducible variety $X$ is \emph{unirational} if the field $\KK(X)$ can be embedded into a purely transcendental extension of the field $\KK$. Geometrically, $X$ is unirational if and only if
there is a dominant rational morphism from some affine space $\AA^m$ to $X$. Clearly, any stably rational variety is unirational. It is proved in \cite[Proposition~5.1]{AFKKZ12013} than any flexible variety is unirational. If $X$ is complete and $\Aut(X)$ acts (generically) 2-transitive on $X$, the variety $X$ is unirational \cite[Corollary~3]{Po2014}.

On the other hand, a homogeneous variety need not be unirational; one may take an elliptic curve as an example. 

\begin{question}
Does the condition $\theta(X)\ge 2$ imply that $X$ is unirational? 
\end{question} 

Now we come to one more geometric property. The classical Chevalley's Theorem~\cite[Theorem~11.2]{Hu1975} claims that any homogeneous space $G/H$ of a linear algebraic group $G$ is a quasi-projective variery. This result holds also for homogeneous spaces of arbitrary algebraic groups; see~\cite[Theorem~5.2.2]{Br22017}. Moreover, it is shown in~\cite[Corollary~4]{Ar12018} that any homogeneous toric variety is quasi-projective. 

\begin{question}
\label{que}
Is any homogeneous variety quasi-projective? 
\end{question}

By the Kleiman-Chevalley criterion~\cite{Kle}, a smooth variety is quasi-projective if and only if every finite subset admits a common affine neighboorhood; later this result was generalized to singular varieties as well. This criterion implies that the answer to Question~\ref{que} is positive provided $\theta(X)=\infty$.

Further, it is proved in \cite[Theorem~A]{Wl} that an irreducible normal variety $X$ admits a closed embedding into a toric variety if and only if every pair of points on $X$ is contained in a common affine neighborhood. Since on a smooth variety every Weil divisor is Cartier, by \cite[Theorem~2]{Hau} we may assume that the ambient toric variety is smooth. So we conclude that if $\theta(X)\ge 2$ then $X$ admits a closed embedding into a smooth toric variety. This property may be considered as a weaker version of quasi-projectivity. 

\section{Appendix}
\label{secapp}

In this section we prove that regular functions separate points on an irreducible algebraic variety $X$ if and only if $X$ is quasi-affine.

\begin{proposition}
\label{propqa}
Let $X$ be an irreducible algebraic variety. The following conditions are equivalent.
\begin{enumerate}
\item
The variety $X$ is quasi-affine.
\item
Regular functions separate points on $X$.
\item 
There are finitely many regular functions that separate points on $X$.
\end{enumerate}
\end{proposition}

\begin{proof}
$(a) \Rightarrow (c)$ If $X$ is a locally closed subset in an affine space, then the coordinate functions separate points on $X$. 

\smallskip

$(c) \Rightarrow (b)$ is obvious. 

\smallskip

$(b)\Rightarrow (a)$. We learned this proof from Dmitri Timashev. For any finitely generated subalgebra $A\subseteq \KK[X]$ we have a dominant morphism $\varphi\colon X\to \Spec A$. By~\cite[Theorem~1.25]{Sha}, there is an open subset $U$ in $\Spec A$ such that for any $y\in U$ all irreducible components of the fiber $\varphi^{-1}(y)$ are of the same dimension $d$ and any other component of any other nonempty fiber of $\varphi$ has dimension $\ge d$. Assume that $d\ge 1$ and fix a point $y_0\in U$. There is a function $f\in\KK[X]$ which is non-constant on some component of the fiber $\varphi^{-1}(y_0)$. If we replace the algebra $A$ by the algebra generated by $A$ and $f$, then some fibers of the new morphism $\varphi\colon X\to \Spec A$ have dimension $<d$. Again by~\cite[Theorem~1.25]{Sha} this is the case for all components of a generic fiber of the new morphism. Continuing this process, we obtain a morphism $\varphi\colon X\to \Spec A$ where fibers over points from an open subset $U$ in $\Spec A$ are finite. 

The preimage $Z$ of $\Spec A\setminus U$ in $X$ is a closed subvariety of dimension $<\dim X$. Again adding regular functions to $A$ we achive that generic fibers of the restriction of $\varphi$ to any irreducible component of $Z$ are finite. Repeating this procedure, we obtain a morphism $\varphi\colon X\to \Spec A$ with finite fibers. 

By Zariski's Main Theorem \cite[Section~III.9]{Mum}, the morphism $\varphi$ can be decomposed as $X\to X'\to \Spec A$, where the first arrow is an open embedding and the second arrow is a finite morphism. Since a finite morphism is an affine morphism and the variety $\Spec A$ is affine, the variety $X'$ is affine as well. So, $X$ is isomorphic to an open subset of an affine variety. 
\end{proof}


%
\end{document}